\documentclass[letterpaper,11pt]{amsart}
\usepackage{amsmath,amssymb,amsthm,pinlabel,tikz,hyperref,mathrsfs,color, thmtools}
\usepackage{verbatim}
\usepackage{float}
\usepackage{caption}
\usepackage{subcaption}
\usepackage{tikz}
\usepackage{enumitem}

\newcommand{\nc}{\newcommand}
\nc{\dmo}{\DeclareMathOperator}
\dmo{\ra}{\rightarrow}
\dmo{\Prob}{\mathbb{P}}
\dmo{\E}{\mathbb{E}}
\dmo{\N}{\mathbb{N}}
\dmo{\Z}{\mathbb{Z}}
\dmo{\Q}{\mathbb{Q}}
\dmo{\R}{\mathbb{R}}
\dmo{\C}{\mathcal{C}}
\dmo{\X}{\mathcal{X}}
\dmo{\U}{\mathcal{U}}
\dmo{\T}{\mathcal{T}}
\dmo{\F}{\mathcal{F}}
\dmo{\w}{\omega}
\dmo{\ev}{\operatorname{ev}}
\dmo{\diam}{\operatorname{diam}}
\dmo{\supp}{\operatorname{supp}}
\dmo{\Cay}{\operatorname{Cay}}
\dmo{\stopping}{\vartheta}
 
\nc{\nt}{\newtheorem}

\nt{theorem}{Theorem}

\newtheorem{thm}{{\bf Theorem}}[section]
\newtheorem{lem}[thm]{{\bf Lemma}}
\newtheorem{cor}[thm]{{\bf Corollary}}
\newtheorem{prop}[thm]{{\bf Proposition}}

\newtheorem{definition}[thm]{Definition}
\numberwithin{equation}{section}

\title{Genericity of contracting geodesics in groups}

\date{\today}
\author{Kunal Chawla}
\address{University of Toronto}
\email{kunal.chawla@mail.utoronto.ca}
\author{Inhyeok Choi}
\address{June E Huh Center for Mathematical Challanges, Korea Institute for Advanced Study}
\email{inhyeokchoi@kias.re.kr}
\author{Giulio Tiozzo}
\address{University of Toronto} 
\email{tiozzo@math.utoronto.ca}

\begin{document}
\begin{abstract}
Let $G$ be a finitely generated group and $\Cay(G, S)$ be the Cayley graph of $G$ with respect to a finite generating set $S$. We characterize the Gromov hyperbolicity of $G$ in terms of the genericity of contracting elements in $\Cay(G, S)$.


\end{abstract}

\maketitle

%%%%%%%%%%%%%%%%%%%%%%%%%%%%%%%%%%%%%%%%%%%%%%%%
%
%							Introduction
%
%%%%%%%%%%%%%%%%%%%%%%%%%%%%%%%%%%%%%%%%%%%%%%%%

\section{Introduction}	\label{sec:introduction}

Let $G$ be a finitely generated group and $\Cay(G, S)$ be the Cayley graph of $G$ with respect to some finite generating set $S$. A central theme of geometric group theory is to relate the geometry of $\Cay(G,S)$ with the dynamics of the group action $G \circlearrowright \Cay(G,S)$. For example, Gromov proposed the notion of hyperbolic groups \cite{gromov1987hyperbolic}, for which strong negative curvature properties of the Cayley graph imply that each element has either finite order or acts as a translation along an infinite geodesic.

An important property of 
any infinite geodesic $\gamma$ in a hyperbolic space is that it is \emph{contracting}: namely, if you project onto $\gamma$ any ball disjoint from $\gamma$, the projection has uniformly bounded diameter. 
We say an element $g$ of $G$ is \emph{contracting} if it has a contracting geodesic, called its axis, along which it acts coarsely by translation.

Recently, there has been a lot of interest in studying the contracting property in spaces that are not hyperbolic (see among others \cite{arzhantseva2015growth}, \cite{sisto2018contracting}, \cite{genevois2019}, 
\cite{CashenMackay}, \cite{yang2020genericity}), since the presence of contracting elements are sometimes sufficient to generalize results from the hyperbolic setting. 

In this paper, we show that genericity of contracting elements can indeed characterize hyperbolic groups, by proving the following dichotomy. 
Let $B(n)$ denote the ball of radius $n$ in $\Cay(G, S)$, centered at the identity. 

\begin{thm}\label{thm:counting}
Let $G$ be a finitely generated group which is not virtually cyclic and let $\Cay(G, S)$ be the Cayley graph of $G$ with respect to a finite generating set $S$. Then we have the following dichotomy: \begin{enumerate}
\item If $G$ is hyperbolic, then there exists $D>0$ such that 
\[
\lim_{n \to \infty} \frac{\#\left\{g \in B(n) : g \,\textrm{is $D$-contracting} \right\}}{\#B(n)} = 1,
\]
where the convergence is exponentially fast.
\item If $G$ is not hyperbolic, then for each $D>0$,\begin{equation}\label{eqn:countingNonHyp}
\lim_{n \to \infty}
\frac{\#\left\{g \in B(n) :  g \,\textrm{is $D$-contracting}\right\}}{\#B(n)} = 0
\end{equation} 
and the decay is exponential in $n$.
\end{enumerate}
\end{thm}

\begin{cor}
A non-virtually-cyclic, finitely generated group is hyperbolic if and only if there exists $D > 0$ such that $D$-contracting elements are generic. \end{cor}

Claim (1) is essentially due to \cite{gekhtman2018counting}: if the group is $\delta$-hyperbolic, every loxodromic element is $D$-contracting with a $D$ that depends only on $\delta$, thus it follows from the genericity of loxodromic elements proved there. Thus, the really new part is the converse: to show that $D$-contracting elements are not generic for non-hyperbolic groups.

The presence of a constant $D$ is necessary: for instance, in $\Z^{2} \ast \Z$, contracting elements are generic but the group is not hyperbolic. In fact, if a finitely generated group has a contracting element, then contracting elements are generic with respect to both the counting measure \cite{yang2020genericity} and any admissible random walk (\cite{sisto2018contracting}, \cite[Theorem A]{choi2022randomII}). 

Our proof strategy is as follows. 
Let $G$ be a non-hyperbolic group. If the group $G$ has no contracting element at all, non-contracting elements are automatically generic in $G$. If not, fix a sufficiently long contracting geodesic segment $\gamma$. Yang proved in \cite[Theorem 2.10]{yang2020genericity} that the axis of a generic element of $G$ almost contains a translate of $\gamma$. We prove an analogous statement for an arbitrarily \emph{non-contracting} segment $\gamma'$ (see Proposition \ref{prop:growthGap}). Namely, given any geodesic segment $\gamma'$, the axis of a generic element of $G$ is nearby the endpoints of a translate of $\gamma'$ in a quantitative sense. This sounds counterintuitive at first, but it follows from the existence of a contracting segment $\gamma$.

\subsection*{Logarithm law for excursions} Finally, in Section \ref{S:exc} we apply our methods to a different problem, namely we study the statistics of maximal excursions of generic geodesics in $G$ into a coarsely geodesically connected subset $H$ (e.g. a quasiconvex subgroup). 

Excursions into cusps of a hyperbolic manifold have been studied for a long time, starting with the celebrated \emph{logarithm law} of Sullivan \cite{Sullivan}, which has been generalized to several contexts.
In the context of Cayley graphs, recently \cite{qing2021excursions} defined a general notion of excursion into a subgroup $H < G$, and established a logarithm law for right-angled Artin groups. 

Here, we generalize this approach to any finitely generated group with a contracting element. Given any geodesic $\gamma$ in the Cayley graph of $G$, define the $K$-coarse excursion of $\gamma$ into $H$ as \[\mathcal{E}_{H,K}(\gamma) := \max_{t \in G} \diam(\gamma \cap \mathcal{N}_K(tH)).\]
    Given an element $g \in G$, define the $K$-coarse excursion of $g$ into $H$, denoted by $\mathcal{E}_{H,K}(g)$, as the maximum of $\mathcal{E}_{K,H}(\gamma)$ over all geodesics from the identity to $g$.
In Theorem \ref{thm:excursions} we prove that such excursions  are generically logarithmic in the length of the word. More precisely, we show: 

\begin{thm}{\label{thm:intro-excursions}}
    Let $G$ be a finitely generated group which is not virtually cyclic, and let $H \subset G$ be a coarsely geodesically connected subset with infinite diameter. Suppose that $G$ has a contracting element which is strongly independent of $H$. Then there exist $K > 0$ such that for any $p>0$ there exists $C_1,C_2$ such that the counting measure $\mathbb{P}_n$ satisfies 
    \[\mathbb{P}_n \left(g \in G: C_1 \leq \frac{\mathcal{E}_{H,K}(g)}{\log n} \leq C_2\right) \geq 1-O(n^{-p}).\]
\end{thm}

This extends to general groups a result from \cite{qing2021excursions} for right-angled Artin groups and from \cite{sisto-taylor-2019} for random walks.

\subsection*{A remark on random walks} We note that a statement analogous to \cite[Theorem 2.10]{yang2020genericity} has been proven for random walks: see \cite[Theorem 3]{baik2020smallest} or \cite[Theorem A]{choi2022randomII}. Using the techniques in these references, one can establish the case (2) of Theorem \ref{thm:counting} for random walks. Together with the genericity of loxodromics in hyperbolic groups due to \cite{maher2018random}, one can complete the analogue of Theorem \ref{thm:counting} for random walks. However, we will not present the details here.

\subsection*{Acknowledgements}
We thank Abdul Zalloum for raising the question and other insights, and Ilya Gekhtman and Sam Taylor for some useful comments on the first draft.

The second author is supported by Samsung Science \& Technology Foundation (SSTF-BA1702-01 and SSTF-BA1301-51) and by a KIAS Individual Grant (SG091901) via the June E Huh Center for Mathematical Challenges at Korea Institute for Advanced Study. This work was initiated while the second author was visiting the University of Toronto, and the second author thanks the other authors for their hospitality during his stay.

The third author is partially supported by grant RGPIN-2017-06521 from NSERC and an Early Researcher Award from the Government of Ontario.

\section{Preliminaries}

We fix a group $G$ with a finite generating set $S$, endowed with the word metric $d = d_{S}$ with respect to $S$. We denote by $e$ the identity element of $G$ and $|\cdot|$ the word norm on $G$, i.e., $|g| := d_{S}(e, g)$. Moreover, if $A \subseteq G$ is a subset, we let $|A| := \sup_{g \in A} |g|$.   

Given two points $x$ and $y$ in $G$, $[x, y]$ denotes an arbitrary geodesic connecting $x$ to $y$.
Likewise, given two subsets $A$ and $B$ in $G$, $[A, B]$ denotes an arbitrary geodesic connecting a point in $A$ to a point in $B$.

Throughout, projections refer to the nearest point projection. More explicitly, given an element $y$ and a subset $A$ of $G$, let $\pi_A(y)$ be the set of points $a \in A$ such that $ d(y,a) = d(y,A)$, which exist as the space is proper. For subsets $Y$ and $Y'$ of $X$, we define \[\begin{aligned}
\pi_A(Y) &= \cup_{y\in Y} \pi_A(y), \\
d_A(Y, Y') &= \diam (\pi_{A}(Y) \cup \pi_{A}(Y')).
\end{aligned}
\]
We now recall the definition of \emph{contracting element}.

\begin{definition}\label{dfn:contracting}
Let $D > 0$ be a constant. 

A geodesic $\gamma$ in $G$ is said to be \emph{$D$-contracting} if for each geodesic $\kappa$ in $G$ with $d(\gamma, \kappa) > D$ we have $\diam(\pi_{\gamma}(\kappa)) < D$.

Given an element $g \in G$, we say that an infinite geodesic $\gamma$ is an \emph{axis} of $g$ if $\gamma$ and the orbit $\{g^{i}\}_{i \in \mathbb{Z}}$ of $g$ are within bounded Hausdorff distance. An element $g \in G$ is said to be \emph{$D$-contracting} if it has a $D$-contracting axis, and it is said to be \emph{contracting} if it is \emph{$D$-contracting}
for some $D> 0$.

Two elements $g, h \in G$ are said to be \emph{$D$-independent} if the projection of one (equivalently, all) of the axes of one element onto the other has diameter smaller than $D$.
The elements are \emph{independent} if they are $D$-independent for some $D > 0.$
\end{definition}

We now recall some facts about contracting geodesics.
Some of these facts are certainly known to the experts, but we collect the statements and some of the proofs here for convenience.

\begin{lem}\label{lem:contractingNbd}
Let $\gamma$ be a $D$-contracting geodesic, $x, y \in G$ and $[x, y]$ be a geodesic segment. If $d_{\gamma}(x, y) \ge D$, 
there exists a subsegment $[x', y']$ of $[x, y]$ 
that satisfies the following. 
\begin{enumerate}
\item $\diam(x'\cup\pi_{\gamma}(x)) < 2D$ and $\diam(y'\cup\pi_{\gamma}(y)) < 2D$.
\item  $\pi_{\gamma}([x, y])$ and $[x', y']$ are within Hausdorff distance $4D$.
\item For any $a \in \pi_{\gamma}(x)$ and $b \in \pi_{\gamma}(y)$, $[a, b]$ (chosen as a subset of $\gamma$) and $[x', y']$ are within Hausdorff distance $10D$.
\end{enumerate}
\end{lem}

\begin{proof} 
Let $N$ be the open $D$-neighborhood of $\gamma$, $\bar{N}$ be its closure and $S = [x, y] \cap \bar{N}$. $S$ is closed, and since $d_{\gamma}(x, y) \ge D$, $S$ is nonempty. Let $x'$ and $y'$ be points in $S$ that are the closest and the farthest from $x$, respectively. In other words, $d(x, x') = \inf_{s \in S} d(x, s)$, $d(x, y') = \sup_{s \in S} d(x, s)$.

Then $[x, x']$ is disjoint from $N$ so $\diam(\pi_{\gamma}([x, x'])) < D$. It follows that \[
\diam(x'\cup \pi_{\gamma}(x)) \le d(x', \pi_{\gamma}(x')) + d_{\gamma}(x', x) < 2D.
\]
Similarly, the diameter of $y' \cup \pi_{\gamma}(y)$ is smaller than $2D$.

Next, we claim that $S$ is $3D$-coarsely connected. If not, there exists a subsegment $[s, t]$ of $[x, y]$ that is longer than $3D$ and is intersecting with $\bar{N}$ exactly at $s$ and $t$. It follows that $[s, t]$ is disjoint from $N$ and $\diam(\pi_{\gamma}([s, t])) < D$. Hence \[
d(s, t) \le d(s, \gamma) + \diam(\pi_{\gamma}([s, t])) + d(\gamma, t) \le 3D,
\]
a contradiction.

Since $S$ is $3D$-coarsely connected, $S$ and $[x', y']$ are within Hausdorff distance $1.5D$.
Next, $S$ is in the $D$-neighborhood of $\pi_{\gamma}([x, y])$, as $d(s, \pi_{\gamma}(s)) \le D$ for each $s \in S$. Now given a point $u$ on $[x, y]$, let $s$ be the nearest point on $S$ from $u$. Then the geodesic segment $\kappa \subseteq [ x, y]$ connecting $s$ to $u$ is either a point (when $u \in \bar{N}$) or disjoint from $N$ (when $u \notin \bar{N}$). In the first case, $\pi_{\gamma}(u)$ is $D$-close to $u = s \in S$. In the second case, $\pi_{\gamma}(u)$ is $D$-close to $\pi_{\gamma}(s)$, and $2D$-close to $s \in S$. In conclusion, $\pi_{\gamma}([x, y])$ is in the $2D$-neighborhood of $S$, and we conclude that $[x', y']$ and $\pi_{\gamma}([x, y])$ are within Hausdorff distance $3.5D$ of each other. 

Meanwhile, since $[x', y']$ is in the $1.5D$-neighborhood of $S$, it is in the $2.5D$-neighborhood of $\gamma$. Let $f : I = [0, M] \rightarrow G$ and $g : J \subseteq \mathbb{R} \rightarrow G$ be the length parameterizations of $[x', y']$ and $\gamma$, respectively. For each $t \in [0, M]$, there exists $s_{t} \in I$ such that $d(f(t), g(s_{t}))< 2.5D$ (when there are many such $s$, just pick one). Then we have $|s_{t} - s_{t'}| = d(g(s_{t}), g(s_{t'})) \le d(f(t) , f(t')) + 5D = |t-t'| + 5D$ and similarly $|s_{t} - s_{t'}| \ge |t-t'| - 5D$. It follows that the map $t \mapsto s_{t}$ is of the form $t \mapsto \pm t + C$ up to an additive error of at most $10D$. In particular, 
 $[x', y']$ and $[a, b]$ have Hausdorff distance at most $10D$.
\end{proof}

\begin{cor}\label{cor:contractingLip}
If $\gamma$ is a $D$-contracting geodesic, then we have
$$d_{\gamma}(x, y) \le d(x, y) + 4D$$ 
for any $x, y \in G$.
\end{cor}

The following is a slight generalization of \cite[Proposition 2.2]{yang2020genericity}.

\begin{lem}\label{lem:nearby}
For each $D>0$ there exists $E =E(D)>D$ that satisfies the following.
Let $\gamma$ be a $D$-contracting geodesic and $x, y$ be points in the $D$-neighborhood of $\gamma$. Then any geodesic segment $[x, y]$ is $E$-contracting and lies in the $E$-neighborhood of $\gamma$.
\end{lem}

\begin{proof}
For some appropriate constants $E_{1} = E_{1}(D)$, $E_{2} = E_{2}(D, E_{1})$ the following argument holds.

Let $x^{\ast}, y^{\ast}$ be points on $\gamma$ such that $d(x, x^{\ast}) < D$ and $d(y, y^{\ast}) < D$. Then $[x^{\ast}, y^{\ast}]$ is a subsegment of a $D$-contracting geodesic and hence $E_{1}$-contracting (\cite[Proposition 2.2 (3)]{yang2020genericity}). 

We now claim that $[x^{\ast}, y^{\ast}]$ and $[x, y]$ are within Hausdorff distance $3D + 12E_{1}$. By Lemma \ref{lem:contractingNbd}, there exists a subsegment $[x', y']$ of $[x, y]$ for $\gamma = [x^{\ast}, y^{\ast}]$ as in Lemma \ref{lem:contractingNbd}. Note that \[
d(x, x') \le d(x, \pi_{\gamma}(x)) + \diam(\pi_{\gamma}(x)\cup x') \le d(x, \gamma) + 2E_{1} \le D + 2E_{1}
\]
and similarly $d(y, y') \le D+ 2E_{1}$. Hence, $[x, y]$ is within Hausdorff distance $D+2E_{1}$ from $[x', y']$ and within Hausdorff distance $D + 2E_{1} + 10E_{1}$ from $[\pi_{\gamma}(x), \pi_{\gamma}(y)]$. Since $\diam(\pi_{\gamma}(x) \cup x^{\ast}) \le 2D$ and $\diam(\pi_{\gamma}(y) \cup y^{\ast}) \le 2D$, we obtain the claim.

Now, since $[x, y]$ is within Hausdorff distance $3D + 12E_{1}$ from an $E_{1}$-contracting geodesic, it is $E_{2}$-contracting (\cite[Lemma 2.15]{arzhantseva2015growth}, \cite[Proposition 2.2 (2)]{yang2020genericity}).
\end{proof}

\begin{definition}[Alignment]\label{dfn:alignment}
For each $i = 1, \ldots, n$, let $x_{i}$ and $y_{i}$ be elements of $G$ and $\kappa_{i}$ be a path connecting $x_{i}$ to $y_{i}$. We say that $(\kappa_{1}, \ldots, \kappa_{n})$ is $C$-aligned if \[
d_{\kappa_{i}}(y_{i}, \kappa_{i+1}) < C,\quad d_{\kappa_{i+1}}\left(x_{i+1}, \kappa_{i}\right) < C
\]
hold for $i=1, \ldots, n-1$.
\end{definition}

Note that the above definition still makes sense if $\kappa_1$ and $\kappa_n$ are points, which we can think of as degenerate segments whose endpoints coincide. 

Given two geodesics $\gamma = [x, y]$ and $\kappa = [x', y']$, we say that $\gamma$ and $\kappa$ are \emph{$D$-fellow traveling} if $d(x, x') < D$, $d(y, y') < D$ and the Hausdorff distance between the two geodesics is smaller than $D$.

\begin{lem}[{\cite[Lemma 3.3]{choi2022random}}] \label{lem:1segment}
For each $D>0$, there exists $E = E(D)>D$ such that the following holds.

Let $\kappa, \gamma$ be $K$-contracting geodesics that connect $x$ to $y$ and $x'$ to $y'$, respectively. Suppose that $(\kappa, x')$ and $(x, \gamma)$ are $D$-aligned. Then $(\kappa, \gamma)$ is $E$-aligned. 
\end{lem}

\begin{lem}[{\cite[Proposition 2.7.(3)]{yang2019statistically}}]\label{lem:concat}
For each $D>0$, there exist $E = E(D) > D$ and $L = L(D) > D$ that satisfy the following.

Let $x, y \in G$ and $\kappa_{1}, \ldots, \kappa_{N}$ be $D$-contracting geodesics longer than $L$. If $(x, \kappa_{1}, \ldots, \kappa_{N}, y)$ is $D$-aligned, then any geodesic segment $[x, y]$ contains subsegments $[x_{1}, y_{1}]$, $\ldots$, $[x_{N}, y_{N}]$, in order from left to right, such that $[x_{i}, y_{i}]$ and $\kappa_{i}$ $E$-fellow travel.
\end{lem}

The following converse of Lemma \ref{lem:concat} is a direct consequence of the triangle inequality, so we omit the proof.

\begin{lem}\label{lem:concatConverse}
For each $D>0$, there exist $E = E(D) > D$ that satisfies the following.
Let $x, y \in G$ and $\kappa_{1}, \ldots, \kappa_{N}$ be geodesics. If the geodesic segment $[x, y]$ contains subsegments $[x_{1}, y_{1}]$, $\ldots$, $[x_{N}, y_{N}]$, in order from left to right, such that $[x_{i}, y_{i}]$ and $\kappa_{i}$ $D$-fellow travel, then $(x, \kappa_{1}, \ldots, \kappa_{N}, y)$ is $E$-aligned.
\end{lem}

\begin{cor}\label{cor:contractingHereditary}
For each $D>0$ there exists a constant $E = E(D)>D$ that satisfies the following.
Let $x, x', y', y$ be points on a geodesic $\kappa$, in order from left to right, and $\gamma$ be a $K$-contracting geodesic. If $(x', \gamma, y')$ is $D$-aligned, then $(x, \gamma, y)$ is $E$-aligned.
\end{cor}

\begin{proof}
Let $E_{1} = E(D)$ and $L_{1} = L(D)$ be as in Lemma \ref{lem:concat}, and let $E_{2} = E(E_{1})$ be as in Lemma \ref{lem:concatConverse}. Finally let $E = E_{2} + L$.

If $\gamma$ is shorter than $L$, then $(x, \gamma, y)$ is automatically $L$-aligned. If not, we first apply Lemma \ref{lem:concat} and deduce that $[x', y']$ contains a subsegment that $E_{1}$-fellow travels with $\gamma$. This is still true for $[x, y]$, and Lemma \ref{lem:concatConverse} tells us that $(x, \gamma, y)$ is $E_{2}$-aligned as desired.
\end{proof}

\section{Counting}

Note that there are non-hyperbolic groups that have no contracting element at all, e.g., $\Z^{n}$ for $n \ge 2$. Since for those Theorem \ref{thm:counting} is trivial, we focus on groups containing at least a contracting element. 

The following lemma is well-known \cite[Lemma 2.12]{yang2019statistically}.

\begin{lem}\label{lem:indepContracting} \label{lem:exponential}
Let $h \in G$ be a contracting element and $E(h)$ be the collection of elements $g \in G$ such that $ghg^{-1}$ and $h$ are not independent. 
Then $E(h)$ is a subgroup of $G$ that is a finite extension of $\langle h \rangle$.

As a consequence, if $G$ is a non-virtually cyclic, finitely generated group with a contracting element $h \in G$, then it contains at least three (in fact, arbitrarily many) pairwise independent contracting elements.
\end{lem}

Recall that the \emph{growth rate} of a group $G$ with a generating set $S$ is 
$$\lambda(G, S) := \lim_{n \to \infty} \frac{1}{n} \log \# 
\{ g \ : \ |g|_S \leq n \}$$
and that a group $G$ has \emph{exponential growth} if $\lambda(G, S) > 1$ for some (equivalently, all) finite generating sets $S$.

\begin{cor}
If $G$ is a non-virtually cyclic, finitely generated group with a contracting element, then $G$ has exponential growth.
\end{cor}

\begin{proof}
By the previous lemma and the ping pong lemma, $G$ contains a free group in two generators, hence it has exponential growth. 
\end{proof}

\begin{prop} \label{prop:not-hyp}
Let $G$ be a finitely generated group. Then either: 
\begin{enumerate}
\item $G$ is hyperbolic; then there exists $D > 0$ such that every geodesic is $D$-contracting;
\item $G$ does not contain any contracting element, or
\item $G$ is not hyperbolic and contains a contracting element. In this case, for any $D > 0$ there exists a geodesic segment $\gamma$ that is not $D$-contracting.
\end{enumerate}
\end{prop}

\begin{proof}
Suppose there exists some $D>0$ such that every geodesic in some Cayley graph of $G$ is $D$-contracting. To show the claim it suffices to argue that $G$ must be hyperbolic, which we do by arguing that all geodesic bigons are uniformly thin \cite[Theorem 1.4]{papasoglu1995strongly}.

Let $\gamma$ and $\kappa$ be two geodesics connecting a point $x \in G$ to another point $y \in G$. If $d(x, y) < D$, then $\diam(\gamma \cup \kappa) < 2D$. If $d(x, y) \ge D$, then Lemma \ref{lem:contractingNbd} asserts that $[\pi_{\gamma}(x), \pi_{\gamma}(y)]=\gamma$ is contained in the $10D$-neighborhood of $\kappa$. Similarly, $\kappa$ is contained in the $10D$-neighborhood and the bigon is $10D$-thin. 
\end{proof}

From now on, we assume that $G$ has independent $D$-contracting elements $h_{1}, h_{2}, h_{3}$. By increasing $D$ if necessary, we have that: \begin{itemize}
    \item $|h_{1}|, |h_{2}|, |h_{3}| \le D$;
    \item $[e, a^{n}]$ is $D$-contracting for each $a \in \{h_{1}, h_{2}, h_{3}\}$ and $n \in \Z$;
    \item $|\pi_{[e, a^{n}]}([e, b^{m}])| \le D$ for each $a, b \in \{h_{1}, h_{2}, h_{3}\}$ such that $a\neq b$ and $n, m \in \Z$.
\end{itemize} We now determine some constants: \begin{itemize}
\item $D_{1} = E(45D)$ be as in Lemma \ref{lem:1segment};
\item $D_{2} = E(D_{1})$, $L_{1} = L(D_{2})$ be as in Lemma \ref{lem:concat};
\item $D_{3} = E(D_{2})$ be as in Lemma \ref{lem:concatConverse};
\item $D_{4} = E(12D_{3})$ be as in Corollary \ref{cor:contractingHereditary};
\item $D_{5} = E(3D_{4})$, $L_{2} = L(D_{4})$ be as in Lemma \ref{lem:concat}.
\end{itemize}
Note that $D <D_{1} < D_{2} < D_{3} < D_{4} < D_{5}$.

Let $L = L_{1} + L_{2}$. By taking powers of $h_{i}$'s if necessary, we may assume that  $L \le |h_{i}| \le L+D$ for $i=1, 2, 3$.

\begin{lem}\label{lem:uniquecontract}
For each $x, y \in G$, one has the bounds $|\pi_{[e, h]}(x)| < 45D$
and $|\pi_{[e, h^{-1}]}(y)| <45D$ for at least one element $h \in \{h_{1}, h_{2}, h_{3}\}$.
\end{lem}

\begin{proof}
Suppose that $|\pi_{[e, h]}(x)| \ge 45D$ for two elements $h\in \{h_1, h_2, h_3\}$, say for $h=h_{1}$ and $h_{2}$. Let $\gamma$ be a geodesic connecting $e$ to $x$. Let $y_{1} \in [e, \pi_{[e, h_{1}]}(x)]$ such that $|y_{1}| = 45D$. Lemma \ref{lem:contractingNbd} tells us that $[e, \pi_{[e, h]}(x)]$ is contained in the $10D$-neighborhood of $\gamma$ as $d_{[e, h_{1}]}(e, x) \ge 45D > D$. Hence, there exists $x_{1} \in \gamma$ such that $d(x_{1}, y_{1}) \le 10D$. Similarly, we can find $y_{2} \in [e, \pi_{[e, h]_{2}}(x)]$ with $|y_{2}| = 45D$ and $x_{2} \in \gamma$ such that $d(x_{2}, y_{2}) < 10D$.

Note that $x_{1}, x_{2}$ are points on the same geodesic $\gamma$ and $|d(e, x_{1}) - d(e, x_{2})| \le 20D$. This implies that $d(x_{1}, x_{2}) < 20D$, and $d(y_{1}, y_{2}) < 40D$. We then have \[\begin{aligned}
d(y_{1}, \pi_{[e, h_{2}]}(y_{1})) &= d(y_{1}, [e, h_{2}]) \le d(y_{1}, y_{2}) < 40D,\\
|\pi_{[e, h_{2}]}([e, h_{1}])| &\ge |\pi_{[e, h_{2}]}(y_{1})|  \ge 45D - d(y_{1}, \pi_{[e, h_{2}]}(y_{1}))\\
&> 45D - 40D > D,
\end{aligned}
\] a contradiction. We conclude that $|\pi_{[e, h]}(x)| < 45D$ for at least two elements $h \in \{h_{1}, h_{2}, h_{3}\}$.

Similarly, $|\pi_{[e,h^{-1}]}(y)| > 45D$ for at least two elements $h\in \{h_1, h_2, h_3\}$. Hence, at least one $h \in \{h_{1}, h_{2}, h_{3}\}$ satisfies both properties.
\end{proof}

We now fix an arbitrary element $f \in G$ for later use. Let $\{ I_n \}_{n \in \mathbb{N}}$
be a sequence of intervals such that 
$I_n \subset \{0,...,n\}$ for each $n$. In the proof of Theorem \ref{thm:counting}, we will pick $I_n = [\eta, \eta + \epsilon] \cdot n \cap \mathbb{N}$.

We now define: \[\begin{aligned}
\mathcal{A} &:= \left\{ g = u_{1} v_{1}fv_{2} u_{2} : \begin{array}{c} (e, u_1 [e, v_{1}], u_{1}v_{1}f[e, v_{2}], g)\,\,\textrm{is $D_{4}$-aligned}, \\v_{1}, v_{2} \in \{h_{1}, h_{2}, h_{3}\}, \,\, |u_{1}| \in I_{|g|} \end{array} \right\}, \\
\mathcal{B} &:= G \setminus \mathcal{A}.
\end{aligned}
\]
Let $\mathcal{B}_{n} := \{ g \in \mathcal{B} : |g|= n\}$. We want to prove:

\begin{prop} \label{prop:growthGap}
Let $G$ be a non-virtually cyclic, finitely generated group with a contracting element. Then there exists $C> 0$ such that for any choice of $f, \{I_n\}_{n\in\mathbb{N}}$ and $M > 10(L+D_{2})+|f|$, we have
$$\# B(e,n) \geq  \#\mathcal{B}_n\cdot \frac{1}{C\#B(e, M)} \left(1+\frac{1}{C  \# B(e,M)}\right)^{\lfloor|I_n|/M\rfloor}$$ 
for any $n \geq 1$. 
\end{prop}

To show this, let us first fix a geodesic segment $\gamma(g) = [e, g]$ for each $g \in \mathcal{B}_{n}$. Then $\gamma(g) = (\gamma(g)_{k})_{k=0}^{n}$ is a sequence in $G$ satisfying \[
d\big(\gamma(g)_{k}, \gamma(g)_{k'}\big) = \big|\gamma(g)_{k}^{-1} \gamma(g)_{k'}\big| = k' - k \quad (0 \le k \le k' \le n).
\] For each $0 \le k \le k' \le n$, let us define: \[
\mathcal{B}_{n, [k, k']} := \big\{ \gamma(g)_{k}^{-1} \gamma(g)_{k'} : g \in \mathcal{B}_{n}\big\}.
\]

\begin{lem} 
For $0=k_{1} \le k_{2} \le \ldots \le k_{m} = n$, we have \[
\prod_{i=1}^{m-1} \big(\# \mathcal{B}_{n,[k_{i}, k_{i+1}]} \big)\ge \#\mathcal{B}_{n}.
\]
\end{lem}

\begin{proof}
Let $f : \mathcal{B}_{n} \rightarrow \prod_{i=1}^{m-1} \mathcal{B}_{n, [k_{i}, k_{i+1}]}$ be the natural map \[
f(g) := ( \gamma(g)_{k_{1}}^{-1} \gamma(g)_{k_{2}}, \ldots, \gamma(g)_{k_{m-1}}^{-1} \gamma(g)_{k_{m}}).
\]
This is an injection. 
Indeed, if $f(g) = f(g') = (w_1, \dots, w_m)$, then by construction 
$$\gamma(g)_{n} = w_1 w_2 \dots w_m = \gamma(g')_{n}$$
which implies $g = g'$.
\end{proof}

Let us set $N_{sep}= 20D_{2}$. Now let $H = \{h_{i}^{-1} h_{j} : i, j \in \{1, 2, 3\}\}$ and pick a maximal subset $\mathcal{B}_{n}^{\circ}$ of $\mathcal{B}_{n}$ whose elements are mutually $(H$, $N_{sep})$-separated.
In other words, we require that $d(x, hy) \ge N_{sep}$ for all distinct $x, y \in \mathcal{B}_{n}^{\circ}$ and $h \in H$. 

\begin{lem}
For any $n \in \mathbb{N}$, one has the inequality
    $$(\# \mathcal{B}_{n}^{\circ}) (7\#B(e, N_{sep})) \ge \#\mathcal{B}_{n}.$$
\end{lem}

\begin{proof}
For any $z \in \mathcal{B}_n$, by maximality there exists $h \in H, x \in \mathcal{B}_{n}^{\circ}$ such that $d(z, h x) \leq N_{sep}$. Otherwise the set $\mathcal{B}_n^{\circ}\cup \{z\}$ would still be $(H, N_{sep})$-separated (which follows from $H = H^{-1}$). Hence 
$$\mathcal{B}_n \subseteq \bigcup_{
\stackrel{h \in H}{ x \in \mathcal{B}_n^{\circ}}} B(h x, N_{sep})$$ 
from which the claim follows since $\# H  =7$. 
\end{proof}

We similarly define the separated version $\mathcal{B}_{n, [k, k']}^{\circ}$ of $\mathcal{B}_{n, [k, k']}$.
By the same reasoning, we have \[(\# \mathcal{B}_{n, [k,k']}^{\circ}) (7\#B(e, N_{sep})) \ge \#\mathcal{B}_{n, [k, k']}.\]

Now given integers $ k_{1} < k_{2} < \ldots < k_{m-1} $ in the interval $I_n/M$, we draw a choice  \[
(g_{1}, \ldots, g_{m}) \in \mathcal{B}^{\circ}_{n, [0, Mk_{1}]} \times \mathcal{B}^{\circ}_{n, [M(k_{1}+1), Mk_{2}]} \times \cdots \times \mathcal{B}^{\circ}_{n, [M(k_{m-1} + 1), n]}.
\]

Let $b_0 := e$. For each $i=1, \ldots, m-1$, there exists a choice $a_{i}, b_{i} \in \{h_{1}, h_{2}, h_{3}\}$ 
such that all of the following hold: \begin{enumerate}
\item  $(b_{i-1}^{-1}, g_{i} [e, a_{i}])$ is $45D$-aligned,
\item $([e, a_{i}], a_{i}f)$ is $45D$-aligned,
\item $(e, a_{i} f [e, b_{i}])$ is $45D$-aligned,
\item $([e, b_{i}], b_{i} g_{i+1})$ is $45D$-aligned.

\end{enumerate}

Indeed, by Lemma \ref{lem:uniquecontract} there exists a choice of $a_i \in \{h_{1}, h_{2}, h_{3}\}$ such that $|\pi_{[e, a_{i}]}(g_i^{-1} b_{i-1}^{-1})|<45D$ and $|\pi_{[e, a_{i}^{-1}]}(f)|<45D$. Similarly, there exists $b_{i}$ that works. Note that our choice of $a_{i}$ depends on $(g_{1}, \ldots, g_{i})$ and $(a_{j}, b_{j})_{j=1}^{i-1}$, while $b_{i}$ depends on $(g_{1}, \ldots, g_{i+1})$, $(a_{j}, b_{j})_{j=1}^{i-1}$ and $a_{i}$. If multiple choices work, we pick the one corresponding to $h_i$ for minimal $i \in \{1,2,3\}$.

After choosing $a_{i}$'s and $b_{i}$'s, we let \[
w_{i} = \left(\prod_{j=1}^{i-1} g_{j} a_{j} f b_{j} \right)\cdot g_{i}, \quad v_{i} = w_{i} a_{i} f
\]
for $i=1, \ldots, m$. We claim that if the above 4 conditions hold for each $i$, then 
\[
\big(e, \,w_{1} [e, a_{1}], \, v_{1}[e, b_{1}], \,\ldots, w_{i}  [e, a_{i}], \, v_{i} [e, b_{i}], \, \ldots, \,v_{m-1}[e, b_{m-1}] , \,w_{m}\big)
\]
is $D_{1}$-aligned. 

By Lemma \ref{lem:1segment}, Condition (2) and (3) for $i$ tell us that $([e, a_{i}], a_{i} f [e, b_{i}])$ is $D_{1}$-aligned; so is $(w_{i}[e, a_{i}], v_{i}[e, b_{i}])$. Also, Condition (4) for $i-1$ and Condition (1) for $i$ guarantee that $([e, b_{i-1}], b_{i-1} g_{i} [e, a_{i}])$, and hence $(v_{i-1} [e, b_{i-1}], w_{i} [e, a_{i}])$, is $D_{1}$-aligned. Let us now check the endpoint cases.
Since $b_0 = e$, Condition (1) for $i = 1$ yields $\pi_{[e, a_1]}(g_1^{-1}) < 45D$ so $(e, g_{1}[e, a_{1}])$ is $45D$-aligned, hence $D_{1}$-aligned. Moreover, Condition (4) with $i = m-1$ yields
$$\diam (\pi_{[v_{m-1}, v_{m-1} b_{m-1}]}(w_{m}) \cup \{v_{m-1} b_{m-1} \}) < 45D,$$ 
which implies that $(v_{m-1}[e, b_{m-1}], w_{m})$ is $D_{1}$-aligned as wanted.

\medskip
Given such a choice of $k_i$'s and $g_i$'s, we define \[
F\left((k_{i})_{i=1}^{m-1}, (g_{i})_{i=1}^{m}\right) := w_{m}= g_{1} a_{1} f b_{1} g_{2} a_{2} f b_{2} \cdots g_{m}.
\]
We also fix indices $k_{0} = -1$ and $k_{m} = n/M$ (possibly a non-integer).

\begin{lem}\label{lem:target}
For any choice $(k_{i}), (g_{i})$, 
the element $F((k_{i})_{i}, (g_{i})_{i})$ lies in $B(e, n)$.
\end{lem}
\begin{proof}
By construction we have $|a_{i}| + |b_{i}| + |f| \leq M$ for $i=1, \ldots, m-1$ and $|g_i| \leq M(k_{i} - k_{i-1}-1)$. Hence 
$$|F((k_{i})_{i=1}^{m-1}, (g_{i})_{i=1}^{m})| \leq \sum_{i = 1}^m M(k_{i} - k_{i-1}-1) + (m-1)M \leq n. \qedhere$$
\end{proof}

\medskip
Given $m$, there are $\binom{|I_n|/M}{m-1}$ ways to pick the sequence $(k_{i})_{i=1}^{m-1}$. Having chosen $(k_{i})_{i}$, note that
\[
\# \mathcal{B}_{n, [0, Mk_{1}]} \cdot \# \mathcal{B}_{n, [Mk_{1}, M(k_{1}+1)]} \cdot \# \mathcal{B}_{n, [M(k_{1} + 1), Mk_{2}]} \cdots \# \mathcal{B}_{n, [M(k_{m-1} + 1), n]} \ge \# \mathcal{B}_{n}.
\]
Recall that \[\begin{aligned}
\#\mathcal{B}_{n, [Mk_{i}, M(k_{i} + 1)]} &\le \#B(e, M), \\
\#\mathcal{B}_{n, [M(k_{i} + 1), Mk_{i+1}]} &\le 7\#B(e, N_{sep}) \cdot \#\mathcal{B}^{\circ}_{n, [M(k_{i} + 1), Mk_{i+1}]}
\end{aligned}
\]for each $i$. Using these inequalities, we deduce that the number of possible $(g_{1}, \ldots, g_{m})$ for a given $(k_{i})_{i}$ is at least \[\begin{aligned}
\prod_{i=1}^{m} \# \mathcal{B}_{n, [M(k_{i-1} + 1), Mk_{i}]}^{\circ} &\ge \frac{1}{(7\#B(e, N_{sep}))^{m}} \prod_{i=1}^{m} \# \mathcal{B}_{n, [M(k_{i-1} + 1), Mk_{i}]} \\
&\ge \frac{\#\mathcal{B}_{n}}{\left(7\#B(e,N_{sep})\#B(e, M)\right)^{m} }.
\end{aligned}
\]
Summing this up, the number of all possible inputs is at least 
\[
\sum_{m=1}^{|I_n|/M} \binom{|I_n| /M}{m-1} \cdot \frac{\#\mathcal{B}_{n}}{(C \#B(e, M))^{m}} \ge \#\mathcal{B}_{n} \cdot \frac{1}{C \#B(e, M)}\left(1 + \frac{1}{C\#B(e, M)}\right)^{|I_n|/ M},
\]
where $C = 7\#B(e, N_{sep})$.

We will be done with the proof of Proposition \ref{prop:growthGap} if we prove the following:
\begin{lem}
    If $F\left((k_i)_{i},(g_i)_i\right) = F\left((k'_i)_{i},(g'_i)_i\right)$, then $(k_i)_i = (k'_i)_i$ and $(g_i)_i = (g'_i)_i$.
\end{lem}

\begin{figure}
    \centering
    \begin{tikzpicture}
    \def\c{1.1}
    \draw (0, 0.2*\c) -- (9*\c, 0.2*\c) -- (9*\c, 0.4*\c) -- (0, 0.4*\c) -- cycle;
    \fill[black!30] (1.5*\c, 0.2*\c) -- (3*\c, 0.2*\c) -- (3*\c, 0.4*\c) -- (1.5*\c, 0.4*\c) -- cycle;
    
    \fill[black!30] (4.5*\c, 0.2*\c) -- (6*\c, 0.2*\c) -- (6*\c, 0.4*\c) -- (4.5*\c, 0.4*\c) -- cycle;
    \fill[black!15] (2*\c, 0.2*\c) -- (2.5*\c, 0.2*\c) -- (2.5*\c, 0.4*\c) -- (2*\c, 0.4*\c) -- cycle;
    \fill[black!15] (5*\c, 0.2*\c) -- (5.5*\c, 0.2*\c) -- (5.5*\c, 0.4*\c) -- (5*\c, 0.4*\c) -- cycle;

    \draw (0, -0.2*\c) -- (9*\c, -0.2*\c) -- (9*\c, -0.4*\c) -- (0, -0.4*\c) -- cycle;
    \fill[black!30] (1.5*\c, -0.2*\c) -- (3*\c, -0.2*\c) -- (3*\c, -0.4*\c) -- (1.5*\c, -0.4*\c) -- cycle;
    
    \fill[black!30] (7.5*\c, -0.2*\c) -- (9*\c, -0.2*\c) -- (9*\c, -0.4*\c) -- (7.5*\c, -0.4*\c) -- cycle;
    
    \fill[black!15] (2*\c, -0.2*\c) -- (2.5*\c, -0.2*\c) -- (2.5*\c, -0.4*\c) -- (2*\c, -0.4*\c) -- cycle;
    \fill[black!15] (8*\c, -0.2*\c) -- (8.5*\c, -0.2*\c) -- (8.5*\c, -0.4*\c) -- (8*\c, -0.4*\c) -- cycle;

    \fill (-0.3*\c, 0) circle (0.07);
    \fill (9.3*\c, 0) circle (0.07);
    \draw[dashed] (-0.3*\c, 0) -- (9.3*\c, 0);

    \end{tikzpicture}
    \caption{Insertion of $a_{i} f b_{i}$'s. Independent of the contracting power of $f$, the endpoints of the darkly filled segments are nearby the horizontal geodesic. It follows that: $g_{1} = g_{1}'$ and $g_{2}'$ cannot come from $\mathcal{B}_{n, [M(k_{2} + 1), Mk_{2}]}$.}
    \label{fig:align}
\end{figure}
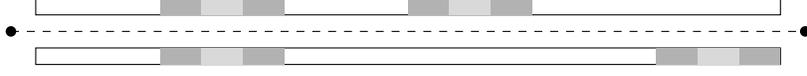

\begin{proof}
    
Let $(k_{i})_{i}, (k_{i}')_{i}$ be given and $(g_{i})_{i}, (g_{i}')_{i}$ be given. 
Let us assume $s = F\left((k_{i})_{i}, (g_{i})_{i}\right) = F\left((k_{i}')_{i}, (g_{i}')_{i}\right)$.

We first discuss the case where $(k_{1}, \ldots, k_{j}) = (k'_{1}, \ldots, k_{j}')$, $(g_{1}, \ldots, g_{j-1}) = (g_{1}', \ldots, g_{j-1}')$ but $g_{j} \neq g_{j}'$. 

The alignment lemma (Lemma \ref{lem:concat}) tells us that $[e, s]$ passes through the $D_{2}$-neighborhood of $p_{0} :=  g_{1} a_{1} f b_{1} \cdots a_{j-1} f$, $p_{1} := p_{0} b_{j-1}$ and $p_{2} := p_{0} b_{j-1} g_{j}$. 
Let $q_{0}, q_{1}, q_{2} \in [e, s]$ be such that $d(p_{i}, q_{i}) \le D_{2}$. Then 
\[\begin{aligned}
d(e, q_{2}) &= d(e, q_{0}) + d(q_{0}, q_{1}) + d(q_{1}, q_{2}) \\
&\ge d(e, q_{0}) + (L-  2 D_{2}) + (|g_{j}| -  2 D_{2}) \\
&\ge d(e, q_{0}) + (L -  2 D_{2}) + (M(k_{j}  - k_{j-1} - 1)-  2 D_{2})
\end{aligned}.
\]
and similarly $d(e, q_{2}) \le d(e, q_{0}) + L + (M(k_{j}  - k_{j-1} - 1)) + 5D_{2}$. 

Meanwhile, $[e, s]$ also passes through the $D_{2}$-neighborhood of $p_{1}' = p_{0} b_{j-1}'$ and $p_{2}'= p_{0} b_{j-1}' g_{j}'$. If we let $q_{2}' \in [e, s] \cap B(p_{2}', D_{2})$, we have the same estimates for $d(e, q_{2}')$, using the fact that $k_{j-1} = k_{j-1}'$, $k_{j} = k_{j}'$. In particular, we have $d(q_{2}, q_{2}') \le 9 D_{2}$, and $d(p_{2}, p_{2}') \leq 11 D_{2}$.

Meanwhile, note that $d(p_{2}, p_{2}') = d(g_{j}, b_{j-1}^{-1} b_{j-1}' g_{j}')$ and $b_{j-1}^{-1} b_{j-1}' \in H$. Since $g_{j}, g_{j}'$ are distinct elements of the $(H, N_{sep})$-separated set $\mathcal{B}_{n, [M(k_{j-1} + 1), Mk_{j}]}^{\circ}$, this cannot happen.

Hence, if $k_{i}$ and $k_{i}'$ match till $i=j$, then $g_{i}$ and $g_{i}'$ also match till $i=j$. 

We now discuss when $(k_{i})_{i} \neq (k_{i}')_{i}$. Say, $k_{i} = k_{i}'$ for $i < j$ but $k_{j} < k_{j}'$. This implies that $g_{i} =g_{i}'$ for $i < j$ and \[
g_{0} a_{1} f b_{1} \cdots a_{j-1} f = g_{0}' a_{1}' f b_{1}' \cdots a_{j-1}' f
\]
holds. We denote this point by $p_{0}$. When $j = 1$, we set $p_{0} = id$; recall $b_{0} = b_{0}' =e $ also.

We know by definition of $a_j$ and $b_j$ that 
$$(e, \ldots,  \,p_{0} [e, b_{j-1}], \,p_{0} b_{j-1} g_{j} [e, a_{j}], p_{0} b_{j-1} g_{j} a_{j} f [e, b_{j}], \ldots, s)$$ 
is $D_{1}$-aligned. By Lemma \ref{lem:concat}, $[e, s]$ contains subsegments \[
[P_{0}, P_{1}], \,\,[P_{2}, P_{3}]\,\, \textrm{and} \,\,[P_{4}, P_{5}],
\]in order from left to right, that $D_{2}$-fellow travel with \[
p_{0} [e, b_{j-1}], \,\,p_{0} b_{j-1} g_{j} [e, a_{j}]\,\, \textrm{and} \,\,p_{0} b_{j-1} g_{j} a_{j} f [e, b_{j}],
\]respectively. When $j = 1$, we instead apply Lemma \ref{lem:concat} to the sequence $(e, g_{1}[e, a_{1}], g_{1}a_{1}f[e, b_{1}], \ldots, s)$ to take subsegments $[P_{2}, P_{3}]$ and $[P_{4}, P_{5}]$ of $[e, s]$ as desired, and then take $P_{0} = P_{1} = e$.

Similarly, since $(e, \ldots, p_{0}[e,  b_{j-1}'], p_{0} b_{j-1}'  g_{j}' [e, a_{j}'], \ldots, s)$ is also $D_{1}$-aligned, $[e, s]$ contains subsegments $[Q_{1}, Q_{2}]$ and $[Q_{3}, Q_{4}]$, in order from left to right, that $D_{2}$-fellow travel with $p_{0}[e,  b_{j-1}']$ and $p_{0} b_{j-1}'  g_{j}' [e, a_{j}']$, respectively. (We again take $Q_{1} = Q_{2} = e$ when $j = 1$.)

We now claim \begin{equation}\label{eqn:2ndCaseInjectivity}
d(e, Q_{2}) - 6D_{2} \le d(e, P_{2}) \le  d(e, P_{5}) \le d(e, Q_{3}).
\end{equation}
For the first inequality we argue for the case $j>1$, as it automatically holds when $j=1$. Since $p_{0} b_{j-1}'$ and $Q_{2}$ are nearby, we have \[\begin{aligned}
d(e, Q_{2}) &\le d(e, p_{0}) + d(p_{0}, p_{0}b_{j-1}') + d(p_{0} b_{j-1}', Q_{2}) \\
&\le |p_{0}| + (L+D) + D_{2}.
\end{aligned}
\]
Next, since $[P_{0}, P_{1}]$ is a subsegment of $[e, P_{2}]$ that fellow travels with $p_{0}[e, b_{j-1}]$, we have \[\begin{aligned}
d(e, Q_{2}) &= d(e, P_{0}) + d(P_{0}, P_{1}) + d(P_{1}, P_{2}) \\
&\ge \big(|p_{0}| - D_{2}\big) + \big(|b_{j-1}| - 2D_{2} \big) \\
&\ge |p_{0}| + L - 3D_{2}.
\end{aligned}
\]
Combining these inequalities yields the first inequality. 

For the last inequality, note that $|g_{j}|' \ge |g_{j}| + M$ due to the assumption $k_{j} < k_{j}'$. Therefore we have\[
\begin{aligned}
d(e, P_{5}) &\le d(e, p_{0}b_{j-1}g_{j}a_{j} fb_{j}) + d(p_{0}b_{j-1}g_{j}a_{j} fb_{j}, P_{5})\\
&\le |p_{0}| + |b_{j-1}| + |g_{j}| + |a_{j}| + |f| + |b_{j}| + D_{2} \\
&\le |p_{0}| + |g_{j}| + M-5D_{2} \\
&\le |p_{0}| + |g_{j}|' - 5D_{2} \\
&\le \big(|p_{0}| - D_{2}\big) + \big(|b_{j-1}'| - 2D_{2}\big) + \big(|g_{j}'| - 2D_{2}) \\
&\le d(e, Q_{1}) + d(Q_{1}, Q_{2}) + d(Q_{2}, Q_{3}) = d(e, Q_{3}).
\end{aligned}
\]

Given Inequality \ref{eqn:2ndCaseInjectivity}, let $Q_{2}'$ be the point on $[e, s]$ with $d(e, Q_{2}') = \min (d(e, Q_{2}), d(e, P_{2}))$. Then $[P_{2}, P_{3}]$ and $[P_{4}, P_{5}]$ are subsegements of $[Q_{2}', Q_{3}]$, in order from left to right, that fellow travel with $p_{0}b_{j-1} g_{j} [e, a_{j}]$ and $p_{0}b_{j-1} g_{j} a_{j }f[ e, b_{j}]$, respectively. Lemma \ref{lem:concatConverse} then tells us that \[
(Q_{2}', p_{0}b_{j-1} g_{j} [e, a_{j}], p_{0}b_{j-1} g_{j} a_{j }f[ e, b_{j}], Q_{3})
\]
is $D_{3}$-aligned. Note also that \[\begin{aligned}
d(Q_{2}', p_{0} b_{j-1}') &\le d(Q_{2}', Q_{2}) + d(Q_{2}, p_{0}b_{j-1}') \le 7D_{2}, \\
d(Q_{3}, p_{0} b_{j-1}' g_{j-1}') &\le D_{2}.
\end{aligned}
\]
By Corollary \ref{cor:contractingLip}, we conclude that \[
(p_{0} b_{j-1}', p_{0}b_{j-1} g_{j} [e, a_{j}], p_{0}b_{j-1} g_{j} a_{j }f[ e, b_{j}], p_{0} b_{j-1}' g_{j-1}')
\]
is $(D_{3} + 7D_{2} + 4D)$-aligned, hence $12D_{3}$-aligned. By the group equivariance we also have that 
\[
(e, (b_{j-1}')^{-1}b_{j-1} g_{j} [e, a_{j}], (b_{j-1}')^{-1} b_{j-1} g_{j} a_{j} f [e, b_{j}], g_{j}')
\]
is $12D_{3}$-aligned.
We now pick $g \in \mathcal{B}_{n}$ such that $g_{j}' = \gamma(g)_{M(k_{j-1}'+1)}^{-1} \gamma(g)_{Mk'_{j}}$. We also set
\[
u_{1} = \gamma(g)_{M(k_{j-1}' + 1)} (b_{j-1}')^{-1} b_{j-1} g_{j}.
\]
Then $[\gamma(g)_{M (k'_{j-1} + 1)}, \gamma(g)_{Mk_{j}'}]$ is a subsegment of a geodesic from $e$ to $g$ and \[
(\gamma(g)_{M(k_{j-1}' + 1)}, u_{1}[e, a_{j}], u_{1} a_{j} f [e, b_{j}], \gamma(g)_{Mk_{j}'})
\]
is $12D_{3}$-aligned. Corollary \ref{cor:contractingHereditary} 
now tells us that 
\[
(e, 
u_1
[e, a_{j}], 
u_1 a_{j} f [e, b_{j}], g)
\]
is $D_{4}$-aligned. Thus, we can write 
$g= u_{1}v_{1}fv_{2}u_{2}$
such that $(e, u_{1}[e, v_{1}], u_{1} v_{1} f [e , v_{2}], g)$ is $D_{4}$-aligned, and moreover, 
$$|u_{1}| \le M(k_{j-1}' + 1) + 2(L+D) + M(k_{j} - k_{j-1} - 1) \le Mk_{j} + 2(L+D) \le \max |I_n|.$$ 
This violates the condition that $g \in \mathcal{B}_{n}$, 
so it cannot happen.
\end{proof}

\medskip

As a consequence, we have \[
\# B(e, n) \ge \#\mathcal{B}_{n} \cdot \frac{1}{C \#B(e, M)}\left(1 + \frac{1}{C\#B(e, M)}\right)^{|I_n|/ M},
\]
which is the desired growth gap, completing the proof of Proposition \ref{prop:growthGap}. 

\begin{proof}[Proof of Theorem \ref{thm:counting}]
Let us prove the second item of the theorem. We assume that $G$ is not  hyperbolic and contains a contracting element, and we fix an arbitrary $K>0$. Let constants $D, L$, and $D_i$ for $1 \leq i \leq 5$ be given as in the paragraph before Lemma \ref{lem:uniquecontract}, and 
let $E = E(D_{5} + K)$ be as in Lemma \ref{lem:nearby}. Recall that if $\gamma$ is a $K$-contracting geodesic and $x, y$ are in the $D_{5}$-neighborhood of $\gamma$, then $[x, y]$ must be $E$-contracting.

By Proposition \ref{prop:not-hyp}, since $G$ is not hyperbolic, there exists an element $f \in G$ which is contracting but not $E$-contracting. 

Now let $I_n = [\eta, \eta + \epsilon]\cdot n \cap \mathbb{N}.$ Disregarding an exponentially decaying fraction of $B(e,n)$, consider $g \in \mathcal{A}_n := \mathcal{A} \cap B(e, n)$, where $\mathcal{A}$ is taken from Proposition \ref{prop:growthGap}. Let $g = u_{1} v_{1} f v_{2} u_{2}$ be the decomposition given by Proposition \ref{prop:growthGap}. Since $(e, u_{1}[e, v_{1}], u_{1} v_{1} f [e, v_{2}], g)$ is $D_{4}$-aligned, $[e, g]$ contains subsegments $[P_{1}, P_{2}]$ and $[P_{3}, P_{4}]$, from left to right, that $D_{5}$-fellow travel with $u_{1}[e, v_{1}]$ and $u_{1} v_{1} f [e, v_{2}]$, respectively. It follows that $|u_{2}| \le (1- \eta) |g| + 8D_{5}$.

For such a $g$, we have the following two cases: \begin{enumerate}[label=(\alph*)]
\item $d([e, u_{2}^{-1}], u_{1}[e, v_{1}]) \le D$;
\item $d([e, u_{2}^{-1}], u_{1}[e, v_{1}]) > D$.
\end{enumerate}
In the first case, let $p \in [e, u_{2}^{-1}]$ and $q \in u_{1}[e, v_{1}]$ be such that $d(p, q) \le D$. Then, using $D \leq D_5$, 
\[\begin{aligned}
d(u_{2}^{-1}, u_{1}) &\le d(u_{2}^{-1}, p) + d(p, u_{1}) \\
&\le \big[d(u_{2}^{-1}, e) - d(p, e) \big] + d(p,  u_{1}) \\
&\le d(u_{2}^{-1}, e) - [d(e, u_{1}) - d(u_{1}, p)] + d(p, u_{1}) \\
&\le |u_{2}| - |u_{1}| + 2 [d(p, q) + d(q, u_{1})] \\
&\le (1-2\eta) |g| + 10(D_{5}+ L).
\end{aligned}
\]

Now, we know by definition that $u_1 \in B(e, (\eta + \epsilon)n)$; moreover, there are at most $3$ choices for each of $v_1, v_2$, and as we just saw $u_2^{-1} \in B(u_1, (1 - 2 \eta + \epsilon) n)$, hence the number of possible values of $g = u_1 v_1 f v_2 u_2$ is at most

\[ 
9 \cdot \# B(e, (\eta + \epsilon)n) \cdot \#B(e, (1 - 2 \eta + \epsilon) n) \lesssim \lambda^{1 - \eta + 2 \epsilon}.
\]

By taking $\eta > 0$ and $\epsilon < \eta/2$, 
by disregarding an exponentially decaying fraction of $B(e,n)$, we can rule out such $g$.

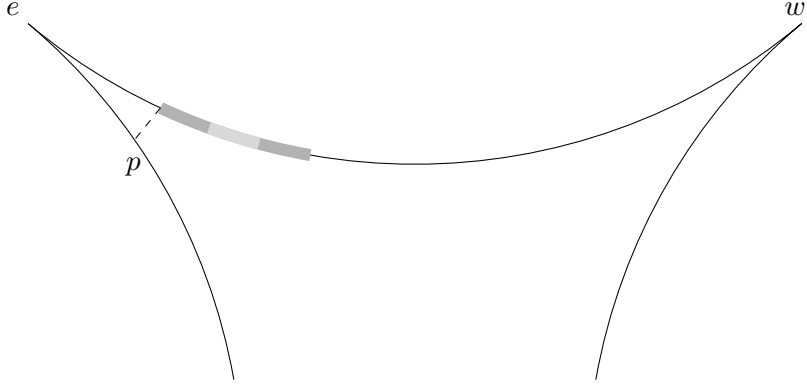
\begin{figure}
    \centering
    \begin{tikzpicture}
    \def\c{1}
        \draw (0,0) arc (-130:-50:8*\c) arc (130:170:8*\c);
        \draw (0, 0) arc (50: 10 : 8*\c);
        \draw[line width=1.6mm, black!30] (1.761354783566717*\c, -1.122106751341376*\c) arc (-115:-100:8*\c);
        \draw[line width=1.6mm, black!15] (2.406139730886963*\c, -1.389185421335443*\c) arc (-110:-105:8*\c);
        
        \draw (-0.2*\c, 0.2*\c) node {$e$};
        \draw (10.2*\c, 0.2*\c) node {$w$};
        \draw (1.4*\c, -1.9*\c) node {$p$};
        \draw[dashed] (1.761354783566717*\c, -1.122106751341376*\c) -- (1.43*\c, -1.53*\c);
    \end{tikzpicture}
    \caption{If the contracting segment $\eta$ on $[e, w]$ is not visible on the axis of $w$, then $\eta$ is near $p \in [e, w^{-1}]$. This allows us to construct $w$ out of the steps on $[e, w]$ till $p$ and then a small linkage word between $p$ and $\eta$, which consists of $(1- \eta + 3\epsilon)n$ letters.}

    \label{fig:my_label}
\end{figure}

Now in the second case, recall the $D$-contracting property of $u_{1}[e, v_{1}]$: since $[e, u_{2}^{-1}]$ is disjoint from its $D$-neighborhood, we have $d_{[e, v_{1}]}(u_{1}^{-1}, u_{1}^{-1}u_{2}^{-1}) \le D$. Moreover, note that
\[\begin{aligned}
d(u_{1}[e, v_{1}], u_{2}^{-1} [e, v_{2}^{-1}]) &\ge |u_{2}| - |u_{1}| - |v_{1}| - |v_{2}|\\
&\ge (1- \eta - \epsilon) n - 4 (D_{5} + L) -|f| \ge D.
\end{aligned}
\]
This implies that the projections of $u_{1}[e, v_{1}]$ and $u_{2}^{-1} [e, v_{2}^{-1}]$ onto each other have diameters smaller than $D$. For a similar reason, the projection of $[e, u_{1}]$ onto $u_{2}^{-1}[e, v_{2}^{-1}]$ has diameter smaller than $D$. We then deduce \[\begin{aligned}
|\pi_{[e, v_{1}]}(u_{1}^{-1}u_{2}^{-1}[e, v_{2}^{-1}])| &\le |\pi_{[e, v_{1}]} (u_{1}^{-1})| + \diam\big(\pi_{[e, v_{1}]}(u_{1}^{-1}[e, u_{2}^{-1}])\big) \\
&+ \diam\big(\pi_{[e, v_{1}]}(u_{1}^{-1}u_{2}^{-1}[e, v_{2}^{-1}])\big) \\
&\le 2D+D_{4}, \\
|\pi_{[e, v_{2}^{-1}])}(u_{2}u_{1}[e,v_{1}])| &\le 
|\pi_{[e, v_{2}^{-1}]}(u_{2})| + \diam 
\big(\pi_{[e, v_{2}^{-1}])}(u_{2}[e, u_{1}])\big) \\&+ \diam 
\big(\pi_{[e, v_{2}^{-1}])}(u_{2}u_{1}[e, v_{1}])\big)\\
&\le 2D + D_{4}.
\end{aligned}
\]

In short, $(u_{2}^{-1} v_{2}^{-1}[e, v_{2}], u_{1}[e, v_{1}])$ is $3D_{4}$-aligned. Note that the $D_{4}$-alignment of $(u_{1}[e, v_{1}], u_{1}v_{1}f[e, v_{2}])$ is already given. Combining these,  we have that  \[
\big( \ldots,\, g^{n} u_{1}[e, v_{1}],\, g^{n}u_{1}v_{1}f[e, v_{2}], \,g^{n+1} u_{1} [e, v_{1}], \,\ldots \big)
\]
is $3D_{4}$-aligned. This means that $g$ has an axis that is $D_{5}$-close to the endpoints of certain translates of $[e, f]$. From this we conclude that $g$ is not $K$-contracting. Indeed, if this axis of $g$ were $K$-contracting, then $[e, f]$ would be $E$-contracting, which is not. Hence, we conclude that $g$ is not $K$-contracting generically.

Denote by $\lambda$ the growth of the group $G$. Thus, we have proven that for any $K > 0$ there exists $\mu < \lambda$, $c > 0$ such that 
\[ \# \{ g \ : \ g \textup{ is }K\textup{-contracting}, |g| = n \} \leq c \mu^n\]
for any $n$.
By summing over all spheres of radius at most $n$ we get the claim. 

The argument so far also leads to the first item of the theorem. Namely, when $G$ is hyperbolic, there exists $K>0$ such that every infinite geodesic on $G$ is $K$-contracting. Setting $f = e$ and running the previous argument, we deduce that a generic element $g \in B(e, n)$ has an axis, i.e. it is loxodromic and not elliptic. Since any axis of $g$ is automatically $K$-contracting, we conclude that a generic element in $B(e, n)$ is $K$-contracting.
\end{proof}

\section{Excursions} \label{S:exc}

In this section, we apply the methods developed in the previous sections to obtain statistical results on the maximal excursion in a subgroup of a finitely generated group $G$.  

The notion of excursion in a subgroup $H < G$ is a generalization to an arbitrary finitely generated group of the well-studied concept of excursion in the cusp of a hyperbolic manifold. 
In the coarser setting of Cayley graphs, excursions into subgroups have been defined in \cite{qing2021excursions}, and their statistical properties have been there established for right-angled Artin groups.

Here, we give a general definition and establish a theorem for excursion in any quasiconvex (in fact, even coarsely geodesically connected) subgroup. 

Following \cite{qing2021excursions}, given a subset $H$ of a finitely generated group $G$, we now define the $K$-coarse excursion of an element $g$ in $H$. We denote as $\mathcal{N}_R(S)$ the $R$-neighbourhood of a set $S$.

\begin{definition}
    Given any geodesic $\gamma$ in the Cayley graph of $G$, define the $K$-coarse excursion of $\gamma$ into $H$ as \[\mathcal{E}_{H,K}(\gamma) := \max_{t \in G} \diam(\gamma \cap \mathcal{N}_K(tH)).\]
    Given an element $g \in G$, define the $K$-coarse excursion of $g$ into $H$, denoted by $\mathcal{E}_{H,K}(g)$, as the maximum of $\mathcal{E}_{K,H}(\gamma)$ over all geodesics from the identity to $g$.
\end{definition}

To recover the notion discussed in \cite{qing2021excursions} we may set $K=0$. 
Say that a contracting element $f$ with axis $\gamma$ is \emph{strongly independent of $H$} if $\sup_{t\in G}\diam(\pi_{\gamma}(tH)) < \infty$. This is stronger than the notion of independence discussed previously, as we quantify over translates of $H$, as opposed to $H$ itself. We say that $H$ is \emph{coarsely geodesically connected} if there exists $R>0$ such that for any $x,y\in H$ there exists a geodesic connecting $x,y$ which lies entirely in $\mathcal{N}_R(H)$.
Let us remark that a quasiconvex set is coarsely geodesically connected. 
For any $n \geq 1$, we denote as $\mathbb{P}_n$ the uniform measure on the ball 
of radius $n$ in the Cayley graph of $(G, S)$.

\begin{thm}{\label{thm:excursions}}
    Let $G$ be a finitely generated group which is not virtually cyclic, and let $H \subset G$ be a coarsely geodesically connected subset with infinite diameter. Suppose that $G$ has a contracting element which is strongly independent of $H$. Then there exist $K > 0$ such that for any $p>0$ there exists $C_1,C_2$ such that the counting measure satisfies 
    \[\mathbb{P}_n \left(g \in G: C_1 \leq \frac{\mathcal{E}_{H,K}(g)}{\log n} \leq C_2\right) \geq 1-O(n^{-p}).\]
\end{thm}

\subsection{Application to RAAGs} 

In a right-angled Artin group, any subgroup which is not conjugated into a join subgroup contains a contracting element (\cite{behrstock2012divergence}, \cite{genevois2019}), 
so we can apply the previous theorem.

\begin{prop}
    Let $G$ be a non-elementary, irreducible right-angled Artin group $G$ equipped with the standard generating set, and let $H$ be an abelian vertex subgroup of rank at least two. Then there exists a contracting element $f$ which is strongly independent of $H$.
\end{prop}
\begin{proof}
    By \cite{behrstock2012divergence}, the group $G$ has a contracting element $f$ for  the action on the Cayley complex with the CAT(0) metric. For some $D>0$, such an element $f$ is $D$-contracting with respect to the action on the Cayley graph \cite[pg. 16]{genevois2019}. If $\sup_{t\in G}\diam(\pi_{\gamma}(tH)) = \infty$, then we can find points $x,y$ in some $tH$ with $d_\gamma(x,y)$ arbitrarily large. By Lemmas \ref{lem:contractingNbd} and then \ref{lem:nearby}, there is a subsegment $[x',y']\subset [x,y]$ such that $d(\pi_\gamma(x),x'), d(\pi_\gamma(y), y') < 2D$, and $[x',y']$ is $E$-contracting for some $E=E(D)$. Observe that $d(x',y') \geq d_\gamma(x,y) - 4D > 3E$. By equivariance of the word metric, we can find arbitrarily long $E$-contracting segments inside $H$, which is impossible as it has rank at least two.
     \end{proof}

In particular, non-elementary irreducible RAAGs contain contracting elements. For any flat, there is a contracting element which is independent of the flat. As flats are geodesically connected, Theorem \ref{thm:excursions} applies to RAAGs, 
recovering the first part of the main result of \cite{qing2021excursions}.

\subsection{Proof of Theorem \ref{thm:excursions}} 
\begin{proof}[Proof of Theorem \ref{thm:excursions}]
We first show that the generic $D_5$-coarse excursion is at least logarithmic. To this end, we apply the results of the previous section with the interval $I_n = [0,1] \cdot n$ and any element $f_n \in H^{-1}\cdot H$, so that $e,f_n$ are contained in a common translate of $H$. We choose $f_n$ with $|f_n| \in [ C_1 \log n-R, C_1 \log n + R]$ for some small $C_1$ to be determined later. 

Recall the sets defined immediately before Proposition \ref{prop:growthGap}: 

\[\begin{aligned}
\mathcal{A} &:= \left\{ g = u_{1} v_{1}f_{|g|}v_{2} u_{2} : \begin{array}{c} (e, u_1 [e, v_{1}], u_{1}v_{1}f_{|g|}[e, v_{2}], g)\,\,\textrm{is $D_{4}$-aligned}, \\v_{1}, v_{2} \in \{h_{1}, h_{2}, h_{3}\}, \,\, |u_{1}| \in I_{|g|} \end{array} \right\}, \\
\mathcal{B} &:= G \setminus \mathcal{A}.
\end{aligned}
\]

Likewise, recall $\mathcal{B}_n := \{g\in \mathcal{B}: |g| = n\}$. For $g \in \mathcal{A}$ of length $n$, alignment guarantees that a subsegment $\gamma$ of $[e,g]$ is $D_5$-close to the endpoints of a translate of $[e,f_n]$. Such a subsegment $\gamma$ must have length at least $C_1\log n - 2D_5 - R$, so we have $\mathcal{E}_{H,D_5}(g) \geq C_1\log n - 2D_5-R$. 

Pick $M_n = 2C_1 \log n$, so that for $n$ sufficiently large we have $M_n > 10(L+D_2) + |f| $. By Proposition \ref{prop:growthGap}, we have 
\[\frac{ \# B(e,n) }{ \#\mathcal{B}_n} \geq  \frac{1}{C\#B(e, 2C_1\log n)} \left(1+\frac{1}{C \# B(e,2C_1 \log n)}\right)^{\left \lfloor \frac{n}{ 2C_1 \log n} \right \rfloor}.\]

We claim that for $C_1$ sufficiently small, the right-hand side grows faster than any polynomial. Since $G$ is finitely generated, we can take $C_1$ such that $B(e,2C_1\log n) \ll n^\alpha$ for some $\alpha < 1$. 

By taking logs, it suffices to show that for any $p$,  \[\frac{ n}{2C_1 \log n} \log \left(1+n^{-\alpha}\right) - p\log n \to \infty \] as $n\to\infty$. Since $\log(1+x) \geq \frac{x}{1+x} $ for $x\geq 0$, the expression is at least $\frac{ n}{2C_1 \log n} \frac{n^{-\alpha}}{2} - p\log n$. Since $\alpha < 1$, this shows the claim.\\

Now we show that a generic $K$-coarse excursion is at most logarithmic for any $K$. 

Let $f$ be a contracting element independent of $H$, with axis $\gamma$, and suppose without loss of generality that it is $D$-contracting. 
Passing to a power if necessary, assume that 
\[|f| > \sup_{t\in G}\diam(\pi_{\gamma}(tH)) +
2(2K + 12D + E + 3A + 2 R)
,\] 
where $R$ is such that $H$ is $R$-coarsely geodesically connected, 
$A=d_{Haus}(\{f^n\}_n, \gamma)$, and $E=E(D)$ is some constant chosen at the end of the proof. 
Pick a constant $M > 10(L+D_2) + |f|$. Let $C_2$ be a large constant, to be chosen later. By Proposition \ref{prop:growthGap}, for any interval $I_n$ of length at least $\frac{1}{2} C_2 \log n$, we have
\[
\frac{\# B(e,n)}{ \#\mathcal{B}_n} \geq  \frac{1}{C\#B(e, M)} \left(1+\frac{1}{C \# B(e,M)}\right)^{\left \lfloor \frac{C_2\log n}{2M}\right \rfloor}.\]

Fix $p > 0$. As $M$ is constant, we can choose $C_2$ sufficiently large so that the right-hand side grows faster than $n^{p+1}$. We can cover $[0,n]$ with at most $\frac{2n}{C_2 \log n}$ intervals of length $\frac{1}{2} C_2 \log n$. Then by the union bound, we know that for all but a $O(n^{-p})$ fraction of elements $g$ of length $n$, any subsegment of length at least $C_2 \log n$ of any geodesic $[e,g]$ contains a further subsegment that $D_5$-fellow travels the endpoints of a translate of $[e,f]$.

For these group elements $g$, we claim that the $K$-coarse excursion of $g$ into $H$ is at most $C_2 \log n$. Let $x,y\in \mathcal{N}_K(tH) \cap [e,g]$ for some $t,K$ and choice of $[e,g]$. Suppose for the sake of contradiction that $d(x,y)$ is larger than $C_2\log n$, so that some subsegment of $[x,y]$ will be $D_5$-close to the endpoints of $[u,uf]$ for some $u\in G$. Let $[x',y']$ be a subsegment of $[x,y]$ such that $d(x', u), d(y', uf) \leq D_5$. Then $(x', u[e,f], y')$ is $2D_5$-aligned. By Corollary \ref{cor:contractingHereditary}, $(x, u[e,f], y)$ is $E$-aligned for some $E=E(2D_5)$. 

Now, let $x^*, y^*$ be points in $tH$ within distance $K$ of $x,y$ respectively. Since the projection onto $u[e, f]$ is coarsely Lipschitz (Corollary \ref{cor:contractingLip}), we know that $(x^*, u[e,f], y^*)$ is $(2K+4D+E)$-aligned. 

Since $tH$ is coarsely geodesically connected, we can choose a geodesic segment $[x^*, y^*]$ which is contained in the $R$-neighbourhood of $tH$ for some $R>0$. From Lemma \ref{lem:contractingNbd}, we know that $[x^*, y^*]$ has a subsegment $[x'',y'']$ such that $d(x'', u[e,f]), d(y'', u[e,f]) < 2D$ and $\pi_{u[e,f]}(x^*,y^*)$ and $[x'',y'']$ are within Hausdorff distance $4D$. Moreover, there exist $\widetilde{x}, \widetilde{y}$ in $t H$ with $d(x'', \widetilde{x}) \leq R, d(y'', \widetilde{y}) \leq R$ 

In particular, we have 
\[\diam(\pi_{u[e,f]}(tH)) \geq d_{u[e,f]}(x'', y'') \geq |f| - 2(2K+8D+E), \] 
from which we deduce that  
\[\diam(\pi_{\gamma}(u^{-1}tH)) > |f| - 2(
2K + 12D + E + 3A + 2 R),\] which contradicts our choice of $f$.

To prove the last estimate, let 
$p = \textup{proj}_{u[e, f]}(x'')$, and $q = \textup{proj}_{ u \gamma}(\tilde{x})$. 
It is enough to prove that $d(p, q) \leq 3 A + 4D + 2 R$, from which the claim follows, as the same estimate holds for $y''$. 

To do so, let $s = \textup{proj}_{u \gamma}(p)$.
Since $q$ is the closest point projection, 
\[ d(q, \tilde{x}) \leq d(s, \tilde{x}) \leq d(s, p) + d(p, x'') + d(x'', \tilde{x}) \leq A + 2 D + R. \]

Moreover, by the triangle inequality 
$d(q, \tilde{x}) \geq d(q, s) - d(s, \tilde{x})$
so $d( q, s) \leq 2 (A + 2D + R)$.
Hence $d(p, q) \leq d(p, s) + d(s, q) \leq 3 A + 4 D + 2 R$
as required.
\end{proof}

\medskip
\bibliographystyle{alpha}
\bibliography{ref}

\end{document}